\documentclass[12pt, a4paper]{amsart}
 \usepackage{amsaddr}
\usepackage[onehalfspacing]{setspace}
 \usepackage{stmaryrd}
\usepackage{fullpage}
\usepackage{amsmath}
\usepackage{epigraph}
\usepackage{amsmath, amscd}
\usepackage{IEEEtrantools}
\usepackage{amssymb, tikz}
\usepackage{mathabx}
\usepackage{mathtools}
\usepackage{comment}

\calclayout

\newtheorem{thm}{Theorem}[section]
\newtheorem{prop}[thm]{Proposition}
\newtheorem{lem}[thm]{Lemma}
\newtheorem{cor}[thm]{Corollary}

\theoremstyle{definition}

\theoremstyle{definition}

\theoremstyle{remark}

\numberwithin{equation}{section}

\def\G{\Gamma}

\newcommand{\R}{\mathbf{R}}  
\newcommand{\Z}{\mathbf{Z}}

\newcommand{\C}{\mathbf{C}}

\newcommand{\SL}{\mathrm{SL}_2 (\mathbf{Z})}

\newcommand{\floor}[1]{\left \lfloor #1 \right \rfloor}

\newcommand{\sgn}{\mathrm{sign}}

\newcommand\SmallMatrix[4]{{\tiny\arraycolsep=0.3\arraycolsep\ensuremath{\begin{pmatrix}#1 & #2 \\ #3 & #4\end{pmatrix}}}}

\def\XXint#1#2#3{{\setbox0=\hbox{$#1{#2#3}{\displaystyle\displaystyle\int}$ }
\vcenter{\hbox{$#2#3$ }}\kern-.6\wd0}}

\usepackage{blindtext}
\usepackage{soul}

\usepackage{etoolbox}

\makeatletter
\patchcmd{\@setauthors}{\MakeUppercase\@author}{}{}{}
\makeatother

\begin{document}

\title{Continued Fractions and Hardy Sums}

\author[A. L\"ageler]{Alessandro L\"ageler}

\maketitle

\begin{abstract}
    The classical Dedekind sums $s(d, c)$ can be represented as sums over the partial quotients of the continued fraction expansion of the rational $\frac{d}{c}$. Hardy sums, the analog integer-valued sums arising in the transformation of the logarithms of $\theta$-functions under a subgroup of the modular group, have been shown to satisfy many properties which mirror the properties of the classical Dedekind sums. The representation as sums of partial quotients has, however, been missing so far. We define non-classical continued fractions and prove that Hardy sums can be expressed as a sums of partial quotients of these continued fractions. As an application, we prove that the graph of the Hardy sums is dense in $\R \times \Z$.
\end{abstract}

\section{Introduction}

The Dedekind sums \begin{equation} \label{classdeddef}
    s(d, c) = \frac{1}{4c} \sum_{k = 1}^{\vert c \vert - 1} \cot \frac{\pi k}{c} \cot \frac{\pi k d}{c}, \; \; \;  (d, c) = 1,
\end{equation}
are ubiquitous objects in mathematics appearing in diverse fields such as physics, geometry, and topology. They are completely determined by the reciprocity law \begin{equation} \label{classreci}
    s(d, c) + s(c, d) = \frac{c^2 + d^2 + 1}{cd} - \frac{1}{4}
\end{equation}
together with the fact that $s(0, 1) = 0$ and $s(d + c, c) = s(d, c)$.

Definition (\ref{classdeddef}) of the Dedekind sums as cotangent-sums is not the only representation of these finite sums. For instance, another formula due to Hickerson \cite{hickerson} expresses $s(d, c)$ in terms of the continued fraction expansion of the rational $\frac{d}{c}$, namely \begin{equation} \label{classhickform}
    s(d, c) = - \frac{1}{4} + \frac{1}{12} \left( \frac{a + d}{c} - \sum_{k = 1}^{2n + 1} (-1)^k a_k \right), \; \; \; 0 < a < c, \; ad \equiv 1 \pmod c,
\end{equation}
where the $a_i \geq 0$ are the partial quotients of the continued fraction expansion of
\begin{equation*}\label{contfrachickform}
    \frac{d}{c} = [0; a_1, ..., a_{2n + 1}] = \frac{1}{a_1 + \frac{1}{\cdots \; + \frac{1}{a_{2n + 1}}}}.
\end{equation*}

The expression (\ref{classhickform}) was used by Hickerson to prove that the set $\{ (d / c, \; s(d, c)) : (d, c) = 1 \}$ is dense in the plane $\R \times \R$, thus proving a conjecture of Rademacher \cite{radegross}.

Dedekind sums first arose in the context of the transformation behavior of $\log \eta$, the logarithm of the Dedekind $\eta$-function $\eta(z) = e^{\frac{\pi i}{12} z} \prod_{n = 1}^\infty (1 - e^{2 \pi i n z})$, under the modular group $\SL$.\footnote{Whenever we talk about the complex logarithm, we shall mean its canonical branch, i.e. with argument in $(- \pi, \pi]$.} It transforms like a weight $\frac{1}{2}$ modular form with a multiplier system. For $A = \SmallMatrix{a}{b}{c}{d} \in \SL$ with $c > 0$, the transformation of $\log \eta(z)$ is as follows: \begin{equation} \label{logetatrans}
    \log \eta (A.z) - \log \eta(z) = \frac{1}{2} \log \left( \frac{cz + d}{i} \right) + \frac{\pi i}{12} \left( \frac{a + d}{c} - 12 \; s(d, c) \right).
\end{equation}

Giving another proof of the reciprocity law (\ref{classreci}), Hardy \cite{hardy} looked at the contour integration of trigonometric functions over a large circle. Changing trigonometric functions in the integrand led him to other finite sums, some of which exhibit reciprocity laws. In this note, we will consider two of these "Hardy sums," namely \begin{equation} \label{Sdefi}
    S(d, c) = \sum_{k = 1}^{c - 1} (-1)^{k + 1 + \floor{\frac{dk}{c}}} \; \; \; \mathrm{for} \; c > 0, \; (d, c) = 1,
\end{equation}
and \begin{equation} \label{S4defi}
    S_4(d, c) = \sum_{k = 1}^{c - 1} (-1)^{\floor{\frac{dk}{c}}} \; \; \; \mathrm{for} \; c > 0, \; (d, c) = 1.
\end{equation}

Both sums $S(d, c)$ and $S_4(d, c)$ in (\ref{Sdefi}) and (\ref{S4defi}) are integer-valued and it is easy to see that $S(-d, c) = -S(d, c)$ and $S(d + 2c, c) = S(d, c)$ and that the same identities also hold for $S_4(d, c)$.

The Hardy sums in (\ref{Sdefi}) and (\ref{S4defi}) are particularly interesting, as they arise analogously to the Dedekind sums as correction factors in the transformation law of the logarithms of $\theta$-functions under subgroups of $\SL$. These $\theta$-functions are also weight $\frac{1}{2}$ forms, but with different multiplier systems. 

More precisely, let \begin{equation} \label{thetadefi}
    \theta(z) = \sum_{n \in \Z} e^{\pi i n^2 z} = \prod_{n = 1}^\infty (1 - e(nz)) \left(1 + e\left( \left( n - 1 / 2 \right) z \right) \right)^2
\end{equation}
be the classical $\theta$-function, where, as usual, $e(z) = e^{2 \pi i z}$. Moreover, define \begin{equation} \label{theta4defi}
    \theta_4(z) = \sum_{n \in \Z} (-1)^n e^{\pi i n^2 z} = \prod_{n = 1}^\infty (1 - e(nz)) \left(1 - e\left( \left( n - 1 / 2 \right) z \right) \right)^2.
\end{equation}

The functions $\theta(z)$ and $\theta_4(z)$ exhibit modular transformations for the subgroups \begin{align}
\begin{split} \label{subgroups}
    &\G_\theta = \left\{ \SmallMatrix{a}{b}{c}{d} \in \SL : a \equiv d, \; b \equiv c \pmod 2 \right\}, \; \mathrm{resp.} \\
    &\G^0(2) = \left\{ \SmallMatrix{a}{b}{c}{d} \in \SL : b \equiv 0 \pmod 2 \right\}
\end{split}
\end{align}
instead of the full modular group $\SL$. Berndt \cite{berndt} proved that \begin{equation} \label{thetatrans}
    \log \theta (A.z) - \log \theta(z) = \frac{1}{2} \log \left( \frac{cz + d}{i} \right) + \frac{\pi i }{4} S(d, c) \; \mathrm{for} \; A = \SmallMatrix{*}{*}{c}{d} \in \G_\theta, \; c > 0,
\end{equation}
and \begin{equation} \label{theta4trans}
    \log \theta_4(A.z) - \log \theta_4(z) = \frac{1}{2} \log\left(\frac{cz + d}{i} \right) - \frac{\pi i}{4} S_4(d, c) \; \mathrm{for} \; A = \SmallMatrix{*}{*}{c}{d} \in \G^0(2), \; c > 0.
\end{equation}

Rademacher \cite{Rademacher} gave the correction factors of the transformation of $\log \theta$ and $\log \theta_4$ in terms of non-trivial linear combinations of Dedekind sums (see also Sitaramachandrarao's elementary approach \cite{spaper}). Moreover, Sczech \cite{sczech} showed by means of the Jacobi Triple Product Identity (\ref{thetadefi}) that the Hardy sums $S(d, c)$ have a representation as cotangent sums, thus resembling the formula (\ref{classdeddef}) for the classical Dedekind sums.

The Hardy sums $S(d, c)$ in (\ref{Sdefi}) make sense for $c, d$ both odd or both even, but the value of the sum will then be equal to zero. Hence, we always implicitly assume that $c + d$ is odd when talking about the sum $S(d, c)$. Similarly, will always implicitly assume that $d$ is odd in $S_4(d, c)$, i.e. that $\SmallMatrix{*}{*}{c}{d}$ lies in $\G^0(2)$. 

Meyer \cite{meyer} proved that the sets $\{ (d / c, \; S(d, c)) : (d, c) = 1, \; c + d \; \mathrm{odd} \}$ and $\{(d / c, \; S_4(d, c)) : (c, d) = 1, \; d \; \mathrm{odd}\}$ are dense in $\R \times \Z$, relying on non-trivial results of Goldberg \cite{goldberg} and Duke, Friedlander, and Iwaniec \cite{dfi}. The goal of this paper is to give another proof of Meyer's result along the lines of Hickerson's proof for the classical Dedekind sums. To this end, we first write the Hardy sums in terms of the partial quotients of a continued fraction expansion.

\begin{thm} \label{mainthm}
Let $d$ be an integer and $c > 0$ such that $(d, c) = 1$. \begin{enumerate}
    \item For $c + d$ odd, let $\frac{d}{c} = \llbracket 2c_0; 2c_1, ..., 2c_n \rrbracket = 2c_0 - \frac{1}{2c_1 - \frac{1}{\cdots - \frac{1}{2c_n}}}$ with $c_0 \in \Z$ and $c_1, ..., c_n$ non-zero integers be the negative continued fraction of $\frac{d}{c}$. The Hardy sum $S(d, c)$ takes the form 
    \begin{equation} \label{Sinthm}
        S(d, c) = - \sum_{k = 1}^n \sgn(c_k).
    \end{equation}
    \item For $d$ odd, let $\frac{d}{c} = [2a_0; a_1, 2a_2, a_3, ..., 2a_{n - 1}, a_n]$ with $a_0 \in \Z$ and $a_1, ..., a_n$ non-zero integers such that $\vert a_k \vert > 1$ for $k = 1, 3, ..., n - 2$. The Hardy sum $S_4(d, c)$ takes the form 
    \begin{equation} \label{S4inthm}
        S_4(d, c) = (a_1 + a_3 + ... + a_n) + \sum_{k = 1}^n (-1)^k \sgn(a_k).
    \end{equation}
\end{enumerate} 
\end{thm}

\textbf{Remark:} The continued fraction expansion $\llbracket 2c_0; 2c_1, ..., 2c_n \rrbracket$ is unique; the expansion $[2a_0; a_1, 2a_2, a_3, ..., a_n]$ from the second part of Theorem \ref{mainthm}, however, is not. The fact that $\frac{d}{c}$ has a continued fraction expansions as claimed in Theorem \ref{mainthm} will be shown in section \ref{contfracsec}. 

\smallskip


Analogously to Hickerson's proof for the density of $\{ (d / c, \; s(d, c)) : (d, c) = 1 \}$ in the plane $\R \times \R$, Theorem \ref{mainthm} will be used to prove both the density of $(d / c, \; S(d, c))$ and $(d / c, \; S_4(d, c))$ in $\R \times \Z$, hence giving a new proof of Meyer's result. 

\begin{thm}[Meyer] \label{meyerthm}
The sets $\{ (d / c, \; S(d, c)) : (d, c) = 1, \; c > 0, \; d + c \; \mathrm{odd} \}$ and $\{ (d / c, \; S_4(d, c)) : (d, c) = 1, \; c > 0, \; d \; \mathrm{odd}\}$ are dense in $\R \times \Z$. 
\end{thm}

Note that $S(d, c) + S_4(d, c)$ makes sense if we restrict the coprime integers $d, c$ to be such that $d$ is odd and $c$ is even (hence, the matrix $\SmallMatrix{*}{*}{c}{d} \in \SL$ lies both in $\G_\theta$ and $\G^0(2)$). The continued fraction expansion of $\frac{d}{c}$ then can be written as $$\frac{d}{c} = \llbracket 2c_0, 2c_1, ..., 2c_n \rrbracket = [2c_0; -2c_1, 2c_2, ..., -2c_n], \; \; n \geq 1 \; \mathrm{odd},$$ with both continued fraction expansions satisfying the assumptions of both parts of Theorem \ref{mainthm}. We thus get the following easy corollary. 

\begin{cor} \label{maincor}
Let $c > 0$ be an even integer and $d$ an odd integer, coprime to $c$. Let $c_0, c_1, ..., c_n$ be integers with $c_1, ..., c_n$ non-zero and $n \geq 1$ odd such that $\frac{d}{c} = \llbracket 2c_0, 2c_1, ..., 2c_n\rrbracket$. We have 
\begin{equation} \label{SS4incor}
S(d, c) + S_4(d, c) = 2 \sum_{\substack{k = 1, \\ k \; \mathrm{odd}}}^{c - 1} (-1)^{\floor{\frac{dk}{c}}} = - 2 \cdot (c_1 + c_3 + ... + c_n).
\end{equation}
\end{cor}

As an application of (\ref{SS4incor}), we can prove the following density theorem, which appears to be new.

\begin{thm} \label{meyerthm2}
The set $\{(d / c, \; S(d, c) + S_4(d, c)) : (d, c) = 1, \; c > 0 \; \mathrm{even}, \; d \; \mathrm{odd} \}$ is dense in $\R \times 2\Z$. 
\end{thm}

Hickerson proved the representation (\ref{classhickform}) of the Dedekind sums $s(d, c)$ using the reciprocity law (\ref{classreci}). Given (\ref{logetatrans}), the reciprocity law can easily be proved by comparing the transformations of $\log \eta$ under $A = \SmallMatrix{*}{*}{c}{d}$ and $AS = \SmallMatrix{*}{*}{d}{-c}$, where $S = \SmallMatrix{0}{-1}{1}{0}$ is the inversion (see, for instance, Grosswald and Rademacher's monograph \cite{radegross}). 

Since $\SmallMatrix{0}{-1}{1}{0} \in \G_\theta$, there is a reciprocity law similar to (\ref{classreci}) for $S(d, c)$ as well, namely \begin{equation} \label{sreci}
    S(d, c) + S(c, d) = \sgn(cd), \; (c, d) = 1, \; c + d \; \mathrm{odd},
\end{equation}
where we set $S(d, - c) = - S(d, c)$ (i.e. we allow negative $c$).

However, since $\SmallMatrix{0}{-1}{1}{0} \notin \G^0(2)$, we should not expect a reciprocity law for $S_4(d, c)$. Formulas that resemble reciprocity laws were given by Meyer \cite{meyer2}, but expressing the right hand side of the reciprocity law by $-1 + S_4(c^2, cd + 1)$. Another reciprocity-type formula for $S_4(d, c)$ can be given by using Theorem \ref{mainthm}, namely \begin{equation}\label{s4reci}
    S_4(d, c) + S_4(c, d) = 2a_0 + 2a_1 + ... + 2a_{n - 1} + a_n - 1,
\end{equation}
where $c > d > 0$ are odd, coprime integers and $\frac{c}{d} = [2a_0; 2a_1, ..., 2a_{n - 1}, a_n]$ with $n \geq 1$ being odd (there is always a unique continued fraction expansion of that kind with the given conditions on $c, d$). However, formula (\ref{s4reci}) is a somewhat unsatisfying reciprocity law, as the right hand side does not involve the integers $c, d$ directly.

The first part of Theorem \ref{mainthm} can be proved by elementary means using the reciprocity formula (\ref{sreci}). As there is at the moment no reciprocity formula for $S_4(d, c)$, an elementary proof of the second part of Theorem \ref{mainthm} seems out of reach.

Instead, we shall use a different approach which starts with the well-known formula \begin{equation} \label{etaeis}
    \log \eta (A.z) - \log \eta(z) = \frac{\pi i}{12} \int_z^{A.z} E_2(w) \; dw, \; A \in \SL,
\end{equation}
where $E_2$ is the Eisenstein series of weight $2$. The reciprocity formula (\ref{classreci}) can thus be interpreted as a consequence of the transformation behavior of $E_2$ under $S$. More importantly, formula (\ref{etaeis}) may be used to give another proof of Hickerson's formula (\ref{classhickform}); see Proposition \ref{eiscontfrac}. There are analogs of (\ref{etaeis}) for both $\log \theta$ and $\log \theta_4$, which allow us to represent their associated Hardy sums in terms of partial quotients.

Since the methods we will use to prove our results are similar for $\log \theta$ and $\log \theta_4$, and, as we said before, the claims for $\log \theta$ can be proved by elementary means as well, we will focus more on $\log \theta_4$, only indicating the proofs for $\log \theta$. 

The rest of the paper is organized as follows: In section \ref{contfracsec}, we will show the necessary properties of continued fractions and their relations to the subgroups $\G_\theta$ and $\G^0(2)$. In section \ref{modformssec}, we use the theory of modular forms to prove the representation of $S(d, c)$ and $S_4(d, c)$ in terms of the continued fraction expansion of $\frac{d}{c}$. In section \ref{densitysec}, we prove Theorems \ref{meyerthm} and \ref{meyerthm2} by means of Theorem \ref{mainthm}. Finally, in section \ref{numericssec}, we give a table of numerical values for $S(d, c)$ and $S_4(d, c)$ with $1 \leq d < c \leq 10$ and the corresponding continued fraction expansions of $\frac{d}{c}$.

\section{Continued Fractions for Subgroups of the Modular Group} \label{contfracsec}

In this section, we prove the facts about continued fractions that we will use later in our proofs. 

A continued fraction is defined to be \begin{equation} \label{contfracdef}
    [a_0; a_1, ..., a_n] = a_0 + \frac{1}{a_1 + \frac{1}{a_2 + \frac{1}{a_3 + \frac{1}{\cdots + \frac{1}{a_n}}}}}
\end{equation}
with $n \geq 0$, $a_0 \in \Z$ and $a_1, ..., a_n$ a sequence of non-zero integers. The integers $a_1, ..., a_n$ are called the partial quotients of the continued fraction. Contrary to convention, we do \textit{not} require $a_k \geq 1$ for $k \geq 1$. Even though it is likely that the continued fractions we consider have already been studied before, the literature on continued fractions is too vast to tell whether this is indeed the case. We thus include all the details of the proof for the convenience of the reader. See \cite{contfrac1}, \cite{contfrac2} for some references on classical continued fractions.

Let $\frac{a}{c}$ be a rational number with $(a, c) = 1$ and let $b, d \in \Z$ be such that $ad - bc = 1$. Continued fraction expansions (c.f.e.) of rational numbers are given by the representation of the matrix $\SmallMatrix{a}{b}{c}{d}$ in the modular group $\SL$ in terms of its generators $T = \SmallMatrix{1}{1}{0}{1}$ and $S = \SmallMatrix{0}{-1}{1}{0}$ or $T$ and $V = TST = \SmallMatrix{1}{0}{1}{1}$. 

Namely, let us write $\SmallMatrix{a}{b}{c}{d} = T^{a_0} V^{a_1} T^{a_2} \cdots V^{a_n} \in \SL$ with $a_i \in \Z$ for $i = 0, ..., n$. We can then read off the c.f.e. of $\frac{a}{c}$ by acting on the cusp $\infty$ by M\"obius transformations: \begin{equation*} 
    \frac{a}{c} = \SmallMatrix{a}{b}{c}{d}.\infty = T^{a_0} V^{a_1} T^{a_2} \cdots V^{a_n}.\infty = a_0 + \frac{1}{a_1 + \frac{1}{a_2 + \frac{1}{... + \frac{1}{a_n}}}} = [a_0; a_1, ..., a_n].
\end{equation*}
Note that since $T.\infty = \infty$, we may write $\SmallMatrix{a}{*}{c}{*}$ for the matrix in (\ref{contfracmatrix}), i.e. the c.f.e. of $\frac{a}{c}$ is independent of multiplication of $\SmallMatrix{a}{b}{c}{d}$ by $T$ on the right. 

We hence showed that every matrix representation in $\SL$ by the generators $T, V$ gives rise to a continued fraction $[a_0; a_1, ..., a_n]$ with $n \geq 0$ odd and vice versa.\footnote{In fact, any continued fraction of the form $\frac{a}{c} = [a_0; a_1, ..., a_{2n'}]$ can be written as $[a_0; a_1, ..., a_{2n'} - 1, 1]$ if $a_{2n'} \neq 1$ or as $[a_0; a_1, ..., a_{2n' - 1} + 1]$ if $a_{2n'} = 1$, so \textit{any} c.f.e. gives a representation of $\SmallMatrix{a}{*}{c}{*}$ in terms of $T, V$, which ends with a power of $V$ (independently of the parity of its length).} 

We may argue similarly for the representation $\SmallMatrix{a}{b}{c}{d} = \pm T^{c_0} S T^{c_1} \cdots T^{c_m} S \in \SL$ with $c_0, c_1, ..., c_m \in \Z$ in terms of the generators $T$ and $S$. The negative continued fraction of $\frac{a}{c}$ can again be obtained by acting on the cusp $\infty$: \begin{equation}\label{contfracmatrix}
\frac{a}{c} = \SmallMatrix{a}{b}{c}{d}.\infty = \pm T^{c_0} S T^{c_1} \cdots T^{c_m} S.\infty = c_0 - \frac{1}{c_1 - \frac{1}{c_2 - \frac{1}{... - \frac{1}{c_m}}}} = \llbracket c_0; c_1, ..., c_m \rrbracket.
\end{equation}

Both the representation of a matrix in the modular group in terms of its generators and the c.f.e. of a rational number arise from Euclidean division. The algorithm to find the coefficients of a (positive) c.f.e. goes as follows: Let $(d, c) = 1$ (to not confuse the resulting rationals in the iterations of the algorithm with the partial quotients, we shall write $\frac{d}{c}$ for the rational number instead of $\frac{a}{c}$ now). Pick $a_0 \in \Z$ so that $\frac{d_0}{c_0} = \frac{d}{c} - a_0 \in [0, 1)$ with $(d_0, c_0) = 1$. If $d_0 = 0$, we are done. Otherwise, we have $\frac{c_0}{d_0} > 1$. Pick $a_1 \geq 1$ an integer such that $\frac{d_1}{c_1} = \frac{d_0}{c_0} - a_1 \in [0, 1)$. Continue this process until one of the fractions $\frac{d_k}{c_k}$ is equal to zero. This algorithm can also be used to find fast rational approximations of irrational numbers. 

The rationals $\frac{d_i}{c_i}$ in the algorithm are positive and the partial quotients $a_i$ are the integer parts of $\frac{c_i}{d_i}$. As such the partial quotients are positive integers. We will see shortly that this will no longer be the case in the algorithms in sections \ref{gthetasec} and \ref{g2sec}. This stems from the fact that the subgroups $\G_\theta$ and $\G^0(2)$ defined in (\ref{subgroups}) are generated by $T^2$ and $S$ resp. $T^2$ and $V$, which will force some or all of the partial quotients to be even integers. The iterations $\frac{d_i}{c_i}$ cannot be taken to be positive then, but only in $(-1, 1)$. Thus, we need to allow negative partial quotients.

However, the c.f.e. of $\G_\theta$- and $\G^0(2)$-type approximate irrational numbers as their $\SL$-counterparts do.


\subsection{$\Gamma_\theta$-type c.f.e.} \label{gthetasec}

Let $c_1, c_2, ...$ be a sequence of non-zero integers and $c_0 \in \Z$. We denote the $k$-th partial fraction of the sequence by $\frac{p_k}{q_k} = \llbracket 2c_0; 2c_1, ..., 2c_k \rrbracket$ with $(p_k, q_k) = 1$. As in (\ref{contfracmatrix}), $p_k$ and $q_k$ may be viewed as the entries of a matrix by $$\frac{p_k}{q_k} = (T^{2c_0} S T^{2c_1} S T^{2c_2} S \cdots T^{2c_k}S).\infty = \SmallMatrix{p_k}{*}{q_k}{*}.\infty.$$ 

Since $\G_\theta$ is generated by $T^2$ and $S$, the matrix $\SmallMatrix{p_k}{*}{q_k}{*} = T^{2c_0} S T^{2c_1} S T^{2c_2} S \cdots T^{2c_k}S$ lies in $\G_\theta$. For this reason we shall call a c.f.e. of the form $\llbracket 2c_0; 2 c_1, ..., 2c_k \rrbracket$ a $\G_\theta$-type c.f.e. For the c.f.e. of $\G_\theta$-type, we shall usually write $\llbracket 2c_1, 2c_2, ..., 2c_n \rrbracket = \llbracket 0; 2c_1, 2c_2, ..., 2c_n \rrbracket$ in the case $c_0 = 0$, dropping the semicolon. 

To find the other two entries of $\SmallMatrix{p_k}{*}{q_k}{*}$, we let the matrix act on the point $0$: $$\SmallMatrix{p_k}{*}{q_k}{*}.0 = (T^{2c_0} S T^{2c_1} S \cdots T^{2c_{k - 1}}S T^{2c_k} S).0 = (S T^{2c_1} S \cdots T^{2c_{k - 1}}S).\infty = \frac{p_{k - 1}}{q_{k - 1}}.$$

Hence, $\SmallMatrix{p_k}{p_{k - 1}}{q_k}{q_{k - 1}} = T^{2c_0} S T^{2c_1} S \cdots T^{2c_k}S \in \mathrm{SL}_2(\Z)$, which implies $p_k q_{k - 1} - p_{k - 1} q_k = 1$. 

The partial fractions $\frac{p_k}{q_k}$ will be useful in finding a sequence of partial quotients $c_i$ of a $\Gamma_\theta$-type c.f.e. leading to an approximation of real numbers in the same way as the classical continued fractions. We start by proving two lemmas, relating the resulting sequence of partial fractions to the sequence of partial quotients. These are analogs for the classical results for the $\SL$-type continued fractions. 

\begin{lem} \label{lemrel}
Let $c_1, c_2, ...$ be a sequence of non-zero integers, $k \geq 1$, and $\frac{p_k}{q_k} = \llbracket 2c_1, ..., 2c_k \rrbracket$ be the $k$-th partial fraction. The partial fractions $\frac{p_k}{q_k}$ satisfy $p_k = 2c_k p_{k - 1} - p_{k - 2}$ and $q_k = 2c_k q_{k - 1} - q_{k - 2}$ for all $k \geq 3$.
\end{lem}

\begin{proof}
By induction. The first three partial fractions are $\frac{p_1}{q_1} = \frac{-1}{2c_1}$, $\frac{p_2}{q_2} = \frac{-2c_2}{4c_1c_2 - 1}$, and $\frac{p_3}{q_3} = \frac{1 - 4c_2c_3}{8c_1c_2c_3 - 2c_1 - 2c_3}$. For $k= 3$, we directly verify the claim.

Suppose the claim for $k - 1 \geq 3$. Let $\frac{p_k'}{q_k'} = \llbracket 2c_2, 2c_3, ..., 2c_k \rrbracket$. We observe that  $\frac{p_k}{q_k} = - \frac{1}{2c_1 - \frac{1}{p_k' / q_k'}} = \frac{-p_k'}{2c_1 p_k' - q_k'}$ and hence $p_k = -p_k'$ and $q_k = 2c_1 p_k' - q_k'$. By the induction hypothesis, we see that $$ \frac{p_k}{q_k} = - \frac{1}{2c_1 - \frac{q_k'}{p_k'}} = - \frac{1}{2c_1 - \frac{2c_k q_{k - 1}' - q_{k - 2}'}{2c_k p_{k - 1}' - p_{k - 2}}} = \frac{p_{k - 2}' - 2c_k p_{k - 1}'}{2c_1c_k p_{k - 1}' - 2c_1 p_{k - 2}' - 2c_k q_{k - 1}' + q_{k - 2}'}.$$

Hence \begin{align*}
    &p_k = p_{k - 2}' - 2c_k p_{k - 1}' = 2c_k p_{k - 1} - p_{k - 2} \; \mathrm{and} \\
    &q_k = 2c_k (2c_1 p_{k - 1}' - q_{k - 1}') - (2c_1 p_{k - 2}' - q_{k - 2}') = 2c_k q_{k - 1} - q_{k - 2}.
\end{align*}
\end{proof}

With Lemma \ref{lemrel} it is then easy to show the following useful fact about the $k$-th partial fractions. 

\begin{lem}\label{lemconv}
Let $c_1, c_2, ...$ be any sequence of non-zero integers, $k \geq 1$, and $\frac{p_k}{q_k} = \llbracket 2c_1, ..., 2c_k \rrbracket$ be the $k$-th partial fraction. Then $\vert p_{k - 1} \vert < \vert p_k \vert$ and $\vert q_{k - 1} \vert < \vert q_k \vert$ for all $k \geq 1$. In particular, $\vert p_k \vert$, $\vert q_k \vert \to + \infty$ as $k \to + \infty$. 
\end{lem}

\begin{proof}
By induction. The proof is essentially the same for $p_k$ and $q_k$, so we will only present the one for $q_k$. We have $q_1 = 2c_1$ and $q_2 = 4c_1c_2 - 1$. Suppose $\vert q_1 \vert \geq \vert q_2 \vert$. Then we would have \begin{equation} \label{qest}
    2 \vert c_1 \vert = \vert q_1 \vert \geq \vert q_2 \vert \geq 4 \vert c_1 \vert - 1,
\end{equation}
but there is no non-zero integer $c_1$ satisfying (\ref{qest}). 

Assume the statement for $k - 1 \geq 2$, i.e. $\vert q_{k - 1} \vert > \vert q_{k - 2} \vert$. Then by Lemma \ref{lemrel}, we have $$\vert q_k \vert \geq 2 \vert c_k \vert \vert q_{k - 1} \vert - \vert q_{k - 2} \vert > \vert q_{k - 1} \vert,$$
as $\vert c_k \vert \geq 1$. 
\end{proof}

With Lemma \ref{lemconv} we are set to prove the approximation of an irrational number by a $\G_\theta$-type c.f.e. and its proof also gives the algorithm by which such an approximation can be computed.

\begin{prop} \label{apprprop}
Let $x \in (-1, 1)$ be irrational. There is a sequence $c_1, c_2, ...$ of non-zero integers such that $ \llbracket 2c_1, 2c_2, ..., 2c_k \rrbracket \to x$ as $k \to + \infty$. 
\end{prop}

\begin{proof}
We construct the sequence $c_1, c_2, ...$ as follows. Since $\vert x \vert < 1$, we have $\frac{1}{\vert x \vert} > 1$. Pick $c_1 \in \Z$ non-zero such that $x_1 = 2c_1 - \frac{1}{x} \in (-1, 1)$ or, equivalently, $x = \frac{1}{2c_1 - x_1} = \llbracket 2c_1, x_1^{-1} \rrbracket$. Then pick $c_2 \in \Z$ non-zero such that $x_2 = 2c_2 - \frac{1}{x_1} \in (-1, 1)$, which again is possible as $\frac{1}{\vert x_1 \vert} > 1$. Equivalently, we may also write $x = \llbracket 2c_1, 2c_2, x_2^{-1} \rrbracket$. Continue this process to get a (unique) finite sequence $c_1, ..., c_k \in \Z - \{0\}$ and $x_1, ..., x_k \in (-1, 1)$ with $x = \frac{1}{2c_1 - \frac{1}{2c_2 - \frac{1}{\cdots - x_k}}} = \llbracket 2c_1, ..., 2c_k, x_k^{-1}\rrbracket$. 

Now observe that \begin{align*}
    \vert \llbracket 2c_1, ..., 2c_k, x_k^{-1} \rrbracket - \llbracket 2c_1, ..., 2c_k \rrbracket \vert &= \left\vert \frac{\frac{1}{x_k} p_k - p_{k - 1}}{\frac{1}{x_k} q_k - q_{k - 1}} - \frac{p_k}{q_k} \right\vert \\
    &= \left\vert \frac{p_k q_{k - 1}- p_{k - 1}q_k}{q_k(\frac{1}{x_k}q_k - q_{k - 1})} \right\vert \\
    &= \frac{1}{\vert q_k \vert \vert \frac{1}{x_k} q_k - q_{k - 1} \vert} \\
    &\leq \frac{1}{\vert q_k \vert} \to 0, \; \mathrm{as} \; k \to + \infty, 
\end{align*}
where we used Lemma \ref{lemconv} and that $\SmallMatrix{p_k}{p_{k - 1}}{q_k}{q_{k - 1}}$ is a unimodular matrix. 
\end{proof}

Note that Proposition \ref{apprprop} gives a \textit{sequence} of integers, which approximates $x$ arbitrarily close. If we were given a $\varepsilon > 0$ and just wanted to find a c.f.e. $\llbracket 2c_1, ..., 2c_n \rrbracket$ in $(x - \varepsilon, x + \varepsilon)$, we could have picked a rational $\frac{a}{c} \in (x - \varepsilon, x + \varepsilon)$ with $c > 0$, $(a, c) = 1$ and $a + c$ odd. Finding $b, d \in \Z$ such that $ad - bc = 1$ yields a matrix $A = \SmallMatrix{a}{b}{c}{d} \in \G_\theta$. Writing the matrix $A$ in terms of the generators $T^2, S$ gives $\llbracket 2c_1, ..., 2c_n \rrbracket$. The downside of this argument is that we would still not know how to construct a c.f.e. with the desired properties, while the proof of Proposition \ref{apprprop} gives an explicit algorithm. Moreover, the algorithm shows that c.f.e. of $\G_\theta$-type are unique (this can be proved rigorously through induction on the length of a c.f.e.).

\smallskip

\subsection{$\G^0(2)$-type c.f.e.} \label{g2sec}

We now turn from the negative continued fractions to the positive continued fraction as in (\ref{contfracdef}). Again, if $a, c \in \Z$ with $(a, c) = 1$ and $a$ an odd integer, we may find $d, b \in \Z$ such that $ad - 2bc = 1$ and $\SmallMatrix{a}{2b}{c}{d} \in \G^0(2)$. Writing $\SmallMatrix{a}{2b}{c}{d} = T^{2a_0} V^{a_1} \cdots T^{2a_{n - 1}} V^{a_n}$ in terms of the generators $T^2$ and $V$ of $\G^0(2)$ gives rise to a c.f.e.: $$\frac{a}{c} = \SmallMatrix{a}{2b}{c}{d}.\infty = T^{2a_0} V^{a_1} \cdots T^{2a_{n - 1}} V^{a_n}.\infty = [2a_0; a_1, ..., 2a_{n - 1}, a_n].$$
We call a c.f.e. of this kind a $\G^0(2)$-type c.f.e.

In the remainder of this section, we present an algorithm for finding the $\G^0(2)$-type c.f.e. of a rational number. First, we prove a technical lemma, relating the length of the continued fraction to the parity of the numerator of the rational it is representing. This will make the description of the algorithm cleaner.

\begin{lem} \label{categorizing}
Let $a_0 \in \Z$ and $a_1, ..., a_n$ be non-zero integers such that $a_{2k}$ is even for all integers $k = 0, 1, ..., \floor{\frac{n}{2}}$. Pick $c > 0$ and $d \in \Z$ with $(c, d) = 1$ such that $\frac{d}{c} = [a_0; a_1, ..., a_n]$. Then $d \equiv n \pmod 2$. 
\end{lem}

\begin{proof}
By induction. The case $n = 0$ is seen from $a_0 \equiv 0 \pmod 2$ and for $n = 1$, we have $\frac{d}{c} = a_0 - \frac{1}{a_1} = \frac{a_0a_1 - 1}{a_1}$, and it is hence clear that $d$ is odd.

Suppose the claim holds for $n - 2 \geq 0$ and write $\frac{d'}{c'} = [a_2, ..., a_n]$. Then $$[a_0; a_1, ..., a_n] = a_0 + \frac{1}{a_1 + \frac{c'}{d'}} = \frac{a_0 (a_1 + c') + d'}{a_1 + c'} = \frac{d}{c},$$
hence, as $a_0$ is even, $d \equiv d' \pmod 2$.
\end{proof}

\begin{prop} \label{contfracg2}
Let $c > 0$ and $d$ an odd integer such that $(c, d) = 1$. There exists a c.f.e. $\frac{d}{c} = [2a_0; a_1, 2a_2, ..., a_n]$ of odd length $n \geq 1$ with $a_0 \in \Z$, $a_1, ..., a_{n} \in \Z - \{ 0 \}$ and $\vert a_1 \vert$, $\vert a_3 \vert$, $...$, $\vert a_{n - 2} \vert > 1$.
\end{prop}

\begin{proof}
The algorithm is similar to the one given in the proof of Proposition \ref{apprprop}. First, pick $a_0 \in \Z$ such that $x_0 = \frac{d}{c} - 2a_0 \in [-1, 1]$. It is not possible that $x_0 = 0$, as then $\frac{d}{c} = [2a_0]$ would be such that $d \equiv 0 \pmod 2$, which we excluded. 

Pick $a_1 \notin \{-1, 0, 1\}$ such that $x_1 = \frac{1}{x_0} - a_1 \in (-1, 1)$. If $x_1 = 0$, we are done. If no such $a_1$ exists, then $x_0 = \pm 1$. Thus, we may pick $a_1 = x_0$ and we are done.

Continue this process to obtain $x_2, x_3, ...$ with $x_{2n} \in [-1, 1]$ and $x_{2n + 1} \in (-1, 1)$ until $x_{2n + 1} = 0$ for an $n \geq 0$. Note that it is not possible that $x_{2n} = 0$, since by Lemma \ref{categorizing} this contradicts the assumption that $d$ is odd.

That the process stops can be seen in the following way. Write $x_1 = \frac{d_1}{c_1}$ with $c_1 > 0$ and $-c_1 < d_1 < c_1$ coprime to $c_1$. Then $x_2 = \frac{c_1 - 2a_2 d_1}{d_1} = \frac{d_2}{c_2}$ with $c_2 > 0$ and $-c_2 \leq d_2 \leq c_2$ coprime to $c_2$. Since $c_2$ divides $d_1$ and $\vert d_1 \vert < c_1$, we have $c_2 < c_1$. In fact, the value of the denominator of $x_1, x_2, ...$ strictly decreases with each step. Hence, the process must come to an end.
\end{proof}

Contrary to the $\G_\theta$-type c.f.e., the continued fraction in Proposition \ref{contfracg2} is \textit{not} unique. Take, say, $\frac{1}{x_0} = \frac{7}{3}$ in the proof of Proposition \ref{contfracg2}. We may either pick $a_1 = 2$ or $a_1 = 3$ to get $x_1 = \frac{1}{x_0} - a_1 \in (-1, 1)$. Indeed, the continued fractions $\frac{3}{7} = [0; 2, 2, 1] = [0; 3, -2, 2]$ both are of the form as claimed in Proposition \ref{contfracg2}. 

Uniqueness can be recovered if we require that the partial quotients $a_1$, $2a_2$, $...$, $a_{n - 2}$, $2a_{n - 1}$ are even. In that case, $\frac{d}{c} = [2a_0; a_1, 2a_2, a_3, ..., 2a_{n - 1}, a_n]$ the denominator $c > 0$ will then have the same parity as $a_n$. In particular, if $c$ is even, the $\G^0(2)$-type c.f.e. of $\frac{d}{c}$ can be transformed into a $\G_\theta$-type c.f.e., namely $$[2a_0; a_1, 2a_2, ..., 2a_{n - 1}, a_{n}] = \llbracket 2a_0; - a_1, 2a_2, ..., -2a_{n - 1}, a_{n} \rrbracket.$$ Proposition \ref{apprprop} is hence also applicable for continued fractions of $\G^0(2)$-type. 


\section{Hardy Sums and (Quasi-)Modular Forms} \label{modformssec}

In this section, we show the analog of (\ref{etaeis}) for $\log \theta$ and $\log \theta_4$ and prove Theorem \ref{mainthm}.

\subsection{Representation of Dedekind Sums as Sums of Partial Quotients.} \label{dederepsec} 

Let \begin{equation} \label{fouriere2}
    E_2(z) = 1 - 24 \sum_{n = 1}^\infty \sigma_1(n) e(nz)
\end{equation}
be the classical Eisenstein series of weight $2$, where $\sigma_1(n) = \sum_{d \vert n } d$. It is well-known that $E_2$ transforms under the action of the modular group by \begin{equation} \label{transforme2}
    E_2(z) \vert_2 \SmallMatrix{a}{b}{c}{d} = (cz + d)^{-2} E_2 \left( \frac{az + b}{cz + d} \right) = E_2(z) + \frac{6}{\pi i} \frac{c}{cz + d}, \; \SmallMatrix{a}{b}{c}{d} \in \SL.
\end{equation}


As mentioned in the introduction, $\log \eta(A.z) - \log \eta(z)$ for $A \in \SL$ is up to a constant equal to the cycle integral of $E_2$; see (\ref{etaeis}). This follows from the classical fact that the logarithmic derivative of $\eta(z)$ is equal to $\frac{\pi i}{12} E_2(z)$ and that $\mathrm{H}^1(\SL, \C) = \{0\}$. More precisely, define \begin{equation} \label{nudefi}
    \nu(A) = \log \eta(A.z) - \log \eta(z) - \frac{\pi i}{12} \int_z^{A.z} E_2(w) dw, \; A \in \SL,
\end{equation}
which does not depend on the choice of $z$. For $A, B \in \SL$, it is easy to see that $\nu(AB) = \nu(A) + \nu(B)$. 

Hence, $\nu : \SL \to (\C, +)$ is a group-homomorphism. But since the generators $S$ and $TS$ of $\SL$ are torsion-elements, there are no non-zero group-homomorphisms $\SL \to (\C, +)$. This proves that $\nu = 0$, i.e. the formula (\ref{etaeis}). 

Equation (\ref{etaeis}) can be used together with the transformation (\ref{logetatrans}) for $\log \eta(z)$ to prove Hickerson's formula (\ref{classhickform}) for the classical Dedekind sums by splitting up the integral of $\displaystyle\int_z^{A.z} E_2(w) \; dw$ into paths from $z_i$ to $T z_i$ or $Vz_i$.


In Hickerson's formula (\ref{classhickform}) we crucially use that the partial quotients of the c.f.e. are positive. But as we will later work with c.f.e. of $\G_\theta$- resp. $\G^0(2)$-type, we shall not require the partial quotients to be positive. In the following proposition, we express the cycle integral $\displaystyle\int_z^{A.z} E_2(w) \; dw$ more generally in terms of a word in $T$ and $V$. The proposition is well-known (see \cite{lang, zagier}). In Appendix \ref{appproof}, we give a proof based on an argument we learned from an unpublished note of W. Duke which does not use the reciprocity law.



\begin{prop} \label{eiscontfrac}
Let $n \geq 1$ be an odd integer and $a_0, ..., a_n$ be integers. Let $A_k = \SmallMatrix{a_k}{(-1)^{k + 1}}{(-1)^k}{0}$ for $k = 0, ..., n$ and write $\SmallMatrix{*}{*}{c_k}{d_k} = A_0 A_1 \cdots A_k$. For $A = \SmallMatrix{a}{b}{c}{d} = A_0 A_1 \cdots A_n$, we have \begin{equation} \label{forminprop}
    \displaystyle\int_z^{A.z} E_2(w) dw = \frac{6}{\pi i} \log \left( \frac{cz + d}{\sgn(c) i} \right)  + \sum_{k = 0}^n (-1)^k a_k + 3 \sum_{k = 1}^n \sgn(c_k d_k).
\end{equation}
\end{prop}

Note that every matrix $A$ of the form $A = T^{a_0} V^{a_1} \cdots T^{a_{n - 1}} V^{a_n} $ may be written as $$A = T^{a_0} S S^{-1} V^{a_1} T^{a_2} \cdots T^{a_{n - 1}} S S^{-1} V^{a_n} = \SmallMatrix{a_0}{-1}{1}{0} \SmallMatrix{a_1}{1}{-1}{0} \cdots \SmallMatrix{a_{n - 1}}{-1}{1}{0} \SmallMatrix{a_n}{1}{-1}{0},$$
i.e. in the form of Proposition \ref{eiscontfrac}. If the integers $a_1, ..., a_n$ are all positive, we have $c_k > 0$ and $\sgn(d_k) = (-1)^{k + 1}$ for $k = 1, ..., n$ in the notation of Proposition \ref{eiscontfrac} . Thus, $3 \sum_{k = 1}^n \sgn(c_k d_k) = 3$. Hence, if we assume that all partial quotients of the c.f.e. of $\frac{a}{c}$ are positive, we get $$\int_z^{A.z} E_2(w) dw = \frac{6}{\pi i} \log \left( \frac{cz + d}{i} \right) + \sum_{k = 0}^n (-1)^k a_k + 3.$$

Comparing this equation to (\ref{logetatrans}) gives Hickerson's formula (\ref{classhickform}).

If the integers $a_1, ..., a_n$ in Proposition \ref{eiscontfrac} are sometimes negative (as they might be if they were partial quotients of a $\G_\theta$- resp. $\G^0(2)$-type c.f.e.), the expression $3 \sum_{k = 1}^n \sgn(c_k d_k)$ in (\ref{forminprop}) does not necessarily evaluate to $3$ anymore. The next proposition is a variant of Proposition \ref{eiscontfrac}, which expresses the right hand side of (\ref{forminprop}) only in terms of the integers $a_1, ..., a_n$ -- given some additional conditions on $a_1, ..., a_n$, but not requiring them to be positive. In particular, these additional conditions are satisfied in the case where $a_1, ..., a_n$ are partial quotients of a $\G_\theta$- resp. $\G^0(2)$-type c.f.e.

\begin{prop} \label{comptermlem}
Let $n \geq 1$ be an odd integer and $a_0, ..., a_n$ be integers such that  $\vert a_2 \vert, ..., \vert a_{n} \vert$ are at least $2$ and $a_1 \neq 0$. Let $A_k = \SmallMatrix{a_k}{(-1)^{k + 1}}{(-1)^k}{0}$ for $k = 0, ..., n$ and write $\SmallMatrix{*}{*}{c_k}{d_k} = A_0 A_1 \cdots A_k$. For $A = \SmallMatrix{a}{b}{c}{d} = A_0 A_1 \cdots A_n$, we have 
$$\displaystyle\int_z^{A.z} E_2(w) dw = \frac{6}{\pi i} \log \left( \frac{cz + d}{\sgn(c) i} \right)  + \sum_{k = 0}^n (-1)^k a_k - 3 \sum_{k = 1}^n (-1)^k \sgn(a_k).$$
\end{prop}

\begin{proof}
We show by induction that $0 < \vert d_k \vert \leq \vert c_k \vert$, $\sgn(c_k) = \sgn(a_1 \cdots a_k)$ and $\sgn(d_k) = (-1)^{k + 1} \sgn(a_1 \cdots a_{k - 1})$ for all $k = 1, ..., n$, which together with Proposition \ref{eiscontfrac} suffices to prove the claim.

Suppose $n = 1$. Then $A_0A_1 = \SmallMatrix{*}{*}{a_1}{1}$, and we see that $\vert d_1 \vert = 1 \leq \vert a_1 \vert = \vert c_1 \vert$, $\sgn(c_1) = \sgn(a_1)$ and $\sgn(d_1) = 1$. 

Suppose the claim holds for $k - 1 < n$. Then $$\SmallMatrix{*}{*}{c_{k - 1}}{d_{k - 1}} \SmallMatrix{a_k}{(-1)^{k + 1}}{(-1)^k}{0} = \SmallMatrix{*}{*}{c_{k - 1} a_k + (-1)^k d_k}{(-1)^{k + 1}c_{k - 1}} = \SmallMatrix{*}{*}{c_k}{d_k},$$
which shows that $\sgn(d_k) = (-1)^{k + 1} \sgn(c_{k - 1}) = (-1)^{k + 1} \sgn(a_1 \cdots a_{k - 1})$. Also, observing that $c_k = \sgn(a_k) \sgn(c_{k - 1}) (\vert a_k c_{k - 1} \vert \pm d_{k - 1})$ gives that $\sgn(c_k) = \sgn(a_1 \cdots a_k)$, since $$\vert a_k c_{k - 1} \vert \pm d_{k - 1} \geq 2 \vert c_{k - 1} \vert - \vert d_{k - 1} \vert \geq \vert c_{k - 1} \vert = \vert d_k \vert > 0.$$
\end{proof}

\subsection{Representation of Hardy Sums as Sums of Partial Quotients.} 

The argument in section \ref{dederepsec} for $\log \eta(z)$ applies to $\log \theta(z)$ and $\log \theta_4(z)$ as well. Let $L(z) = 2E_2(z) - E_2\left( \frac{z}{2} \right)$, which is a modular form of weight $2$ for $\G^0(2)$. It is easy to see that it is invariant under the $\vert_2$-operator for $\G^0(2)$ using $\SmallMatrix{a}{b}{2c}{d}.\frac{z}{2} = \frac{1}{2} \SmallMatrix{a}{2b}{c}{d}.z$. The modular form $L(z)$ is the only modular form of weight $2$ for $\G^0(2)$ up to scalar multiplication. Its Fourier-expansion is \begin{equation} \label{fourierL}
    L(z) = 1 + 24 \sum_{n = 1}^\infty \sigma_1^\mathrm{odd}(n) e^{\pi i n z},
\end{equation}
where $\sigma_1^\mathrm{odd}(n) = \sum_{d \vert n, \; d \; \mathrm{odd}} d$. 

The space of modular forms of weight $2$ for $\G_\theta$ is also one-dimensional and its basis element is given by $$R(z) = L(z + 1) = 1 + 24 \sum_{n = 1}^\infty (-1)^n \sigma_1^\mathrm{odd}(n) e^{\pi i n z}.$$

It follows from the Jacobi Triple Product Identity that \begin{equation} \label{exptheta}
    \theta(z) = e\left( \frac{1}{24} \int_\infty^z (E_2(w) - R(w)) dw \right) \; \mathrm{and} \; \theta_4(z) = e \left( \frac{1}{24} \int_\infty^z (E_2(w) - L(w)) dw \right).
\end{equation}
Note that due to the branch cut of the complex logarithm we cannot simply take the logarithm on both sides of (\ref{exptheta}) to write the logarithm of $\theta$-functions as integrals over quasi-modular forms. However, we can take the logarithmic derivative. For example, the logarithmic derivative $\frac{\theta_4'(z)}{\theta_4(z)}$ is equal to $\frac{\pi i}{12} \left( E_2(z) - L(z) \right)$. 

Similar to the case of $\log \eta(z)$, we define $\phi(A) = \log \theta_4(A.z) - \log \theta_4(z) - \frac{\pi i}{12} \displaystyle\int_z^{A.z} E_2(w) dw$. By differentiation in $z$ it easy to see that the expression $\phi(A)$ is constant in $z$ and thus that $\phi$ is a group-homomorphism $\G^0(2) \to (\C, +)$. Similarly, one shows that for $B \in \G_\theta$ the function $\psi(B) = \log \theta(B.z) - \log \theta(z) - \frac{\pi i}{12} \displaystyle\int_z^{B.z} E_2(w) dw$ is a well-defined group-homomorphism $\G_\theta \to (\C, +)$.

Unlike $\nu : \SL \to (\C, +)$ in (\ref{nudefi}), the group-homomorphisms $\phi$, $\psi$ need not be identically zero anymore, as their domain is a subgroup of $\SL$ and not the full modular group itself. However, the vector-spaces of group-homomorphisms $\G^0(2) \to (\C, +)$ and $\G_\theta \to (\C, +)$ are one-dimensional, which allows us to determine the form of $\phi$ and $\psi$ precisely.

\begin{lem} \label{grouplem}
\begin{enumerate}
    \item There is a unique group-homomorphism $\phi: \G^0(2) \to (\C, +)$, up to scalar multiplication, given by $$\phi(A) = \int_z^{A.z} L(w) dw, \; A \in \G^0(2).$$
    The homomorphism $\phi$ acts on the generators $T^2, V$ of $\G^0(2)$ by $\phi(T^2) = 2$ and $\phi(V) = 2$. 
    
    \item There is a unique group-homomorphism $\psi : \G_\theta \to (\C, +)$, up to scalar multiplication, given by $$\psi(B) = \int_z^{B.z} R(w) dw, \; B \in \G_\theta.$$
    The homomorphism $\psi$ acts on the generators $T^2, S$ of $\G_\theta$ by $\psi(T^2) = 2$ and $\psi(S) = 0$. 
\end{enumerate}
\end{lem}

\begin{proof}
Let $A \in \G^0(2)$. As $L(z)$ is invariant under the $\vert_2$-action of $\G^0(2)$, the integral $\phi(A) = \displaystyle\int_z^{A.z} L(w) dw$ is zero under differentiation by $z$, hence independent on $z$. It is easily checked that $\phi$ is indeed a group-homomorphism. Furthermore, it is clear by the Fourier expansion of $L(z)$ in (\ref{fourierL}) that $\phi(T^2) = \displaystyle\int_z^{T^2.z} L(w) dw = 2$. 

Now suppose that $\phi_0$ is any group homomorphism $\G^0(2) \to (\C, + )$. Since $$2\phi_0(-A) = \phi_0((-A)^2) = \phi_0(A^2) = 2\phi_0(A) \; \; \mathrm{for \; all} \; A \in \G^0(2),$$ we have $\phi_0(-A) = \phi_0(A)$. 

Note that $T^2 V^{-1} T^2 = -V$. Hence, $$\phi_0(V) = \phi_0(-V) = \phi_0(T^2 V^{-1} T^2) = 2\phi_0(T^2) - \phi_0(V),$$
which implies that $\phi_0(T^2) = \phi_0(V)$. This proves that any group-homomorphism is uniquely determined by its value $\phi_0(T^2)$, hence up to a constant equal to $\phi(A)$.\footnote{Another way to find the value of $\displaystyle\int_z^{V.z} L(w) dw$ would be to note that $V.(-1 + i) = 1 + i$. Since the integral is constant in $z$, we may well specialise to the point $z = -1 + i$ to get $\displaystyle\int_{-1 + i}^{1 + i} L(w) dw = 2$.}

The proof for the subgroup $\G_\theta$ is similar, noting that $\psi(S^2) = \psi(-I) = 0$ and hence $\psi(S) = 0$. 
\end{proof}

With Lemma \ref{grouplem}, we may easily prove the analog of (\ref{etaeis}) for the logarithms of the $\theta$-functions. 

\begin{prop} \label{thetaeis}
For $A \in \G^0(2)$ the logarithm of $\theta_4(z)$ satisfy \begin{equation} \label{jacobidentity}
    \log \theta_4(A.z) - \log \theta_4(z) = \frac{\pi i}{12} \int_z^{A.z} (E_2(w) - L(w)) dw.
\end{equation}

For $B \in \G_\theta$, the logarithm of $\theta (z)$ satisfy $$\log \theta(B.z) - \log\theta(z) = \frac{\pi i}{12} \displaystyle\int_z^{B.z} (E_2(w) - R(w))dw.$$ 
\end{prop}

\begin{proof}
Let $A \in \G^0(2)$. As we already saw, the function $$\phi(A) = \log \theta_4(A.z) - \log \theta_4(z) - \frac{\pi i}{12} \int_z^{A.z} E_2(w) dw$$ is independent of $z$ and a group homomorphism $\G^0(2) \to (\C, +)$. By Lemma \ref{grouplem} it is thus equal to a scalar multiple of $\displaystyle\int_z^{A.z} L(w) dw$. Setting $A = T^2$, we see that $$\phi(A) = - \frac{\pi i}{12} \int_z^{A.z} L(w) dw.$$

The proof for $\log \theta$ is similar.
\end{proof}

As seen in section \ref{contfracsec}, the representation of matrices in the groups $\G^0(2)$ and $\G_\theta$ leads to c.f.e. of a certain type. We may use Propositions \ref{eiscontfrac} and \ref{thetaeis} to express $\log \theta_4(A.z) - \log \theta_4(z)$ in terms of the matrix representation of $A = \SmallMatrix{a}{2b}{c}{d}$ as product of the generators $T^2$, $V$ in $\G^0(2)$. 

More precisely, suppose $c > 0$ be an integer and that $-c < d < c$ an odd integer coprime to $c$. It thus has a c.f.e. of $\G^0(2)$-type $\frac{d}{c} = [0; a_1, 2a_2, ..., 2a_{n - 1}, a_n]$ as in Proposition \ref{contfracg2}, i.e. where $a_1, a_2, ..., a_n$ are non-zero integers such that $\vert a_1 \vert, \vert a_3 \vert, ..., \vert a_{n - 2} \vert > 1$. Note that $- \frac{d}{c} = [0; -a_1, -2a_2, ..., -2a_{n - 1}, -a_n]$ and hence $\SmallMatrix{d}{*}{-c}{*} = \pm V^{-a_1} T^{-2a_2} \cdots T^{-2a_{n - 1}} V^{-a_n}$ is a matrix in $\G^0(2)$. Its inverse $\SmallMatrix{*}{*}{c}{d} = \pm V^{a_n} T^{2a_{n - 1}} \cdots T^{2a_2} V^{a_1}$ then satisfies the assumptions on the partial quotients in Proposition \ref{comptermlem}. 

\begin{prop} \label{thetacontfrac}
\begin{enumerate}
    \item Let $n$ be an odd integer and $A = \SmallMatrix{*}{*}{c}{d} = T^{2a_0} V^{a_1} T^{2a_2} \cdots V^{a_n} \in \G^0(2)$ with $a_1, ..., a_n$ being non-zero integers and $\vert a_k \vert > 1$ for $k > 1$ an odd integer. The cocycle $\displaystyle\int_z^{A.z} (E_2(w) - L(w)) dw$ is equal to $$\frac{6}{\pi i} \log \left( \frac{cz + d}{\sgn(c) i} \right) - 3(a_1 + a_3 + ... + a_n) - 3 \sum_{k = 1}^n (-1)^k \sgn(a_k).$$   
    \item Let $n$ be an integer and $A = \SmallMatrix{*}{*}{c}{d} = T^{2c_0} S T^{2c_1} S \cdots T^{2c_n} S \in \G_\theta$ with $c_1, ..., c_n$ being non-zero integers. The cocycle is given by $$\displaystyle\int_z^{A.z} (E_2(w) - R(w)) dw = \frac{6}{\pi i} \log \left( \frac{cz + d}{\sgn(c) i} \right) - 3 \sum_{k = 1}^n \sgn(c_k).$$
\end{enumerate}
\end{prop}

\begin{proof}
For the first part, recall that we may write $T^{2a_k} V^{a_{k + 1}} = T^{2a_k} \SmallMatrix{0}{-1}{1}{0} \SmallMatrix{0}{1}{-1}{0} V^{a_{k + 1}} = \SmallMatrix{2a_k}{-1}{1}{0} \SmallMatrix{a_{k + 1}}{1}{-1}{0}$ for $k = 0, ..., n - 1$. The claim then follows by Proposition \ref{comptermlem} and Lemma \ref{grouplem}.

For the second part, write $\pm A = T^{2c_0} S T^{2c_1} S^{-1} \cdots T^{2c_n} S^{(-1)^n}$ and note that $T^{2c_k} S^{(-1)^k} = \SmallMatrix{(-1)^k 2 c_k}{(-1)^{k + 1}}{(-1)^k}{0}$, i.e. $$\pm A = \SmallMatrix{2c_0}{-1}{1}{0} \SmallMatrix{-2c_1}{1}{-1}{0} \cdots \SmallMatrix{(-1)^n 2c_n}{(-1)^{n + 1}}{(-1)^n}{0}.$$ 
Since $\displaystyle\int_z^{A.z} (E_2(w) - R(w)) dw = \displaystyle\int_z^{(-A).z} (E_2(w) - R(w)) dw$, we may again use Proposition \ref{comptermlem} and Lemma \ref{grouplem} to prove the claim. 
\end{proof}

With Proposition \ref{thetacontfrac} we wrote the cocycle of the logarithm of the $\theta$-functions in terms of the matrix representation as words in $T^2, \; S$ resp. $T^2,\; V$ (i.e. in c.f.e. of $\G_\theta$- resp. $\G^0(2)$-type). Theorem \ref{mainthm} immediately follows from the next proposition, which collects our previous results to give the representation of the Hardy sums in terms of the partial quotients of c.f.e. of $\G_\theta$- resp. $\G^0(2)$-type. 

\begin{prop} \label{proofprop} 
\begin{enumerate}
    \item Let $d, c$ be coprime integers of opposite parity with $c > 0$ and assume that $-c < d < c$. Pick $a, b \in \Z$ such that $ad - bc = 1$. Pick $c_0 \in \Z$ and non-zero integers $c_1, ..., c_n$ with $\SmallMatrix{a}{b}{c}{d} = \pm T^{2c_0} S T^{2c_1} S \cdots T^{2c_n} S$. The Hardy sum $S(d, c)$ can be evaluated as $$S(d, c) = - \sum_{k = 1}^n \sgn(c_k).$$

    \item Let $c > 0$ and $d$ be an odd integer coprime to $c$ such that $-c < d < c$ . Pick $a, b \in \Z$ such that $ad - 2bc = 1$. Pick $a_0 \in \Z$ and non-zero integers $a_1, ..., a_n$ such that $\vert a_k \vert > 1$ for $k > 1$ an odd integer and $\SmallMatrix{a}{2b}{c}{d} = T^{2a_0} V^{a_1} \cdots T^{2a_{n - 1}} V^{a_n}$. The Hardy sum $S_4(d, c)$ can be evaluated as $$S_4(d, c) = a_1 + a_3 + ... + a_n + \sum_{k = 1}
   ^n (-1)^k \sgn(a_k).$$
\end{enumerate}
\end{prop}

\begin{proof}
To prove (1), note that (\ref{thetatrans}), Proposition \ref{thetaeis} and Proposition \ref{thetacontfrac} imply that \begin{align*} 
    \frac{1}{2} \log \left(\frac{c z + d}{i} \right) + \frac{\pi i}{4} S(d, c) &= \frac{\pi i}{12} \int_z^{A.z} (E_2(w) - R(w)) dw \\
    &=\frac{1}{2} \log \left( \frac{c z + d}{i} \right) - \frac{\pi i}{4} \sum_{k = 1}^n \sgn(c_k).
\end{align*}

To prove (2), use (\ref{theta4trans}), Proposition \ref{thetaeis} and Proposition \ref{thetacontfrac} to obtain \begin{align*}
    &\frac{1}{2} \log \left( \frac{cz + d}{i} \right) - \frac{\pi i}{4} S_4(d, c) \\
    &= \frac{\pi i}{12} \int_{z}^{A.z} (E_2(w) - L(w)) dw \\
    &= \frac{1}{2} \log \left( \frac{cz + d}{i} \right) - \frac{\pi i}{4} (a_1 + a_3 + ... + a_n) - \frac{\pi i}{4} \sum_{k = 1}^n (-1)^k \sgn(a_k).
\end{align*}
\end{proof}

\section{Density of Hardy Sums} \label{densitysec}

In this section, we prove Meyer's Theorem \ref{meyerthm} that the sets $\{ (d / c, \; S(d, c)) : c + d \; \mathrm{odd}, c > 0, \; (d, c) = 1 \}$ and $\{ (d / c, \; S_4(d, c)) : c > 0, \; d \; \mathrm{odd}, \; (d, c) = 1 \}$ are dense in $\R \times \Z$, as well as the new density Theorem \ref{meyerthm2}.

The arguments to prove the claim are quite short, but differ between $S(d, c)$ and $S_4(d, c)$. We will present them in separate subsections. 

\subsection{Proof of the Density of the Hardy Sums $S(d, c)$} \label{densityS}

Let $(x, m) \in \R \times \Z$ and $\varepsilon > 0$. We need to show that there are coprime integers $d, c$ with $c > 0$ and $d + c$ odd such that $$\left\vert x - \frac{d}{c} \right\vert < \varepsilon \; \mathrm{and} \; S(d, c) = m.$$ Since $S(d + 2c, c) = S(d, c)$, we may assume w.l.o.g. that $x \in (-1, 1)$. 

By Proposition \ref{apprprop}, we may find non-zero integers $c_1, c_2, ..., c_n$ such that $\frac{p_n}{q_n} = \llbracket 2c_1, ..., 2c_n \rrbracket$ lies in $(x - \frac{\varepsilon}{2}, x + \frac{\varepsilon}{2})$. Pick $M \in \Z$ such that $M + m = S(p_n, q_n)$. If $M = 0$, there is nothing more to show. Thus assume that $M \neq 0$. 

Now construct the following rational in $(x - \varepsilon, x + \varepsilon)$: Pick $c_{n + 1} \geq 1$ big enough such that $$\left\vert \llbracket 2c_1, ..., 2c_n, \sgn(M) 2c_{n + 1} \rrbracket - \frac{p_{n}}{q_{n}} \right\vert \leq \frac{1}{\vert q_{n} \vert (2 \vert c_{n + 1} \vert \vert q_{n} \vert - \vert q_{n - 1}\vert)} < \frac{\varepsilon}{2 \vert M \vert},$$
where we made use of Lemma \ref{lemconv}. Denote $\llbracket 2c_1, ..., 2c_n, \sgn(M) 2c_{n + 1},\rrbracket$ by $\frac{p_{n + 1}}{q_{n + 1}}$. Then pick $c_{n + 2} \geq 1$ such that $$\left\vert \llbracket 2c_1, ..., 2c_n, \sgn(M) 2c_{n + 1}, \sgn(M) 2c_{n + 2} \rrbracket - \frac{p_{n + 1}}{q_{n + 1}} \right\vert < \frac{\varepsilon}{2\vert M \vert}, \; \mathrm{etc.}$$

Continuing this process $\vert M \vert$-times, we get a rational $$\frac{p_{n + \vert M \vert}}{q_{n + \vert M \vert}} = \llbracket 2c_1, ..., 2c_n, 2 c_{n + 1}', ..., 2 c_{n + \vert M \vert}' \rrbracket,$$
where $c_k' = \sgn(M) c_k$ for $k = n + 1, ..., n + \vert M \vert$. In particular, $\sgn(c_k') = \sgn(M)$. 

The rational $\frac{p_{n + \vert M \vert}}{q_{n + \vert M \vert}}$ lies in $(x - \varepsilon, x + \varepsilon)$, as $$\left\vert x - \frac{p_{n + \vert M \vert}}{q_{n + \vert M \vert}} \right\vert \leq \left\vert x - \frac{p_{n}}{q_{n}} \right\vert + \sum_{k = 1}^{\vert M \vert} \left\vert \frac{p_{n + k - 1}}{q_{n + k - 1}} - \frac{p_{n + k}}{q_{n + k}} \right\vert < \varepsilon$$
by construction. Moreover, it follows from Theorem \ref{mainthm} that \begin{align*}
    S(p_{n + \vert M \vert}, q_{n + \vert M \vert}) &= - \sum_{k = 1}^n \sgn(c_k) - \sum_{l = 1}^{\vert M \vert} \sgn(c_{n + l}') \\
    &= m + M - \sgn(M) \vert M \vert \\
    &= m.
\end{align*} 

\subsection{Proof of the Density of the Hardy Sums $S_4(d, c)$ and of $S_4(d, c) + S(d, c)$} 

To prove Theorem \ref{meyerthm2}, we will prove a stronger result than $\{ (d / c, S(d, c) + S_4(d, c)) : d \; \mathrm{odd}, \; c > 0 \; \mathrm{even} \}$ being dense in $\R \times 2\Z$, namely that $$\{ (d / c, \; S(d, c) + S_4(d, c), \; S_4(d, c)) : d \; \mathrm{odd}, \; c > 0 \; \mathrm{even} \} \subset \R \times 2 \Z \times (2\Z + 1)$$ is a dense subset.

For $S(d, c) + S_4(d, c)$ to be well-defined, we need $\frac{d}{c}$ with $c > 0$ and $(d, c) = 1$ to be such that $\SmallMatrix{*}{*}{c}{d} \in \G_\theta \cap \G^0(2)$. In other words, $$\frac{d}{c} = \llbracket 2c_0; 2c_1, ..., 2c_n \rrbracket = [2c_0; -2c_1, 2c_2, -2c_3, ..., -2c_n]$$ should be such that $d$ is odd and $c + d$ is odd, i.e. $c$ is even. By Lemma \ref{categorizing}, this can only happen if $n$ is odd. 

In this case, the Hardy sums are \begin{align*}
    &S(d, c) = - \sum_{k = 1}^n \sgn(c_k) \; \mathrm{and} \; \\
    &S_4(d, c) = -2(c_1 + c_3 + ... + c_n) + \sum_{k = 1}^n \sgn(c_k) = -2(c_1 + c_3 + ... + c_n) - S(d, c)
\end{align*}
by Theorem \ref{mainthm}. It is clear that $S(d, c) + S_4(d, c) \equiv 0 \pmod 2$. With the next lemma, we can also deduce that $S_4(d, c)$ will have to be odd in this case.

\begin{lem} \label{lengthlem}
Let $c_1, ..., c_n$ be non-zero integers and $\frac{d}{c} = \llbracket 2c_1, ..., 2c_n \rrbracket$ for $(d, c) = 1$ and $c > 0$. Then $$S(d, c) \equiv n \pmod 2.$$
\end{lem}

\begin{proof}
Immediately from Theorem \ref{mainthm}.
\end{proof}

\begin{prop} \label{mainthm2prop}
The set $$\{(d / c, \; S(d, c) + S_4(d, c), \; S_4(d, c)) : d \; \mathrm{odd}, \; c > 0 \; \mathrm{even}, \; (d, c) = 1\}$$ is dense in $\R \times 2\Z \times (2\Z + 1)$.
\end{prop}

\begin{proof}
Let $(x, m_1, m_2) \in \R \times 2\Z \times (2\Z + 1)$ and $\varepsilon > 0$. Again, we need to show that there are coprime integers $d, c$ with $c > 0$ even and $d$ odd such that $$\left\vert x - \frac{d}{c} \right\vert < \varepsilon, \; \; S(d, c) + S_4(d, c) = m_1 \; \; \mathrm{and} \; \; S_4(d, c) = m_2.$$ As before, we may assume w.l.o.g. that $x \in (-1, 1)$.

Let $\frac{p_n}{q_n} = \llbracket 2c_1, ..., 2c_n \rrbracket$ be such that $\left\vert x - \frac{p_n}{q_n} \right\vert < \frac{\varepsilon}{3}$ and $S(p_n, q_n) = m_1 - m_2$, which is possible by what we proved in section \ref{densityS}. Since $m_1 - m_2$ is odd, Lemma \ref{lengthlem} tells us that $n$ must be odd as well. 

Let $M \in \Z$ be such that $S(p_n, q_n) + S_4(p_n, q_n) = m_1 + 2M$. If $M = 0$, there is nothing to show, since $$S(p_n, q_n) + S_4(p_n, q_n) = m_1 = m_1 - m_2 + S_4(p_n, q_n).$$ Thus, suppose that $M \neq 0$. 

Pick $c_{n + 1} > 0$ big enough such that $$\left\vert \frac{p_n}{q_n} - \llbracket 2c_1, ..., 2c_n, - \sgn(M) 2c_{n + 1}  \rrbracket \right\vert < \frac{1}{2 \vert c_{n + 1} \vert \vert q_n \vert - \vert q_{n - 1} \vert} < \frac{\varepsilon}{3}$$
and let $\frac{p_{n + 1}}{q_{n + 1}} = \llbracket 2c_1, ..., 2c_n, - \sgn(M) 2c_{n + 1}  \rrbracket$. Setting $$\frac{p_{n + 2}}{q_{n + 2}} = \llbracket 2c_1, ..., 2c_n, - \sgn(M) 2c_{n + 1} , 2M \rrbracket,$$ we see by Lemma \ref{lemrel} that $$\left\vert \frac{p_{n + 1}}{q_{n + 1}} - \frac{p_{n + 2}}{q_{n + 2}} \right\vert < \frac{1}{\vert q_{n + 2} \vert} < \frac{1}{2 c_{n + 1} \vert q_{n + 1} \vert - \vert q_n \vert} < \frac{\varepsilon}{3}.$$

But hence $\vert x - \frac{p_{n + 2}}{q_{n + 2}} \vert < \varepsilon$ and $$S(p_{n + 2}, q_{n + 2}) + S_4(p_{n + 2}, q_{n + 2}) = m_1 + 2M - 2M = m_1$$
by Corollary \ref{maincor}.

Finally, note that $S(p_n, q_n) = S(p_{n + 2}, q_{n + 2})$, giving that $S_4(p_{n + 2}, q_{n + 2}) = m_2$. 
\end{proof}

To finish the proof that $\{(d / c, \; S_4(d, c)) : d \; \mathrm{odd}, c > 0, \; (d, c) = 1 \}$ is dense in $\R \times \Z$, we turn to the possibility of $c > 0$ being odd. A continued fraction expansion of $\G^0(2)$-type of the form $$\frac{d}{c} = [2a_1, 2a_2, ..., 2a_{n - 1}, a_n] = \llbracket -2a_1, 2a_2, ..., 2a_{n - 1}, -a_n\rrbracket$$
is of $\G_\theta$-type if and only if $a_n$ is even. Hence, assuming that $a_n$ is odd, we see that $c > 0$ ought to be odd. This distinction will be implicit in the proof of the next proposition.

\begin{prop} \label{meyerthmprop}
The set $\{ (d / c, \; S_4(d, c)) : d \; \mathrm{odd}, \; c > 0, \; (d, c) = 1 \}$ is dense in $\R \times \Z$. 
\end{prop}

\begin{proof}
Let $(x, m) \in (-1, 1) \times \Z$ and $\varepsilon > 0$.  We need to find coprime integers $d, c$ with $c > 0$ and $d$ odd such that \begin{equation} \label{denscond}
    \left\vert \frac{d}{c} - x \right\vert < \varepsilon \; \; \mathrm{and} \; \; S_4(d, c) = m.
\end{equation}

If $m$ is odd, there is a $\frac{d}{c}$ satisfying the condition (\ref{denscond}) by Proposition \ref{mainthm2prop}. Thus, suppose that $m$ is even.

Again by Proposition \ref{mainthm2prop}, pick $\frac{p_n}{q_n} = \llbracket 2c_1, ..., 2c_n \rrbracket$ with $n$ odd such that $S_4(p_n, q_n) = m + 1$ and $\vert x - \frac{p_n}{q_n} \vert < \frac{\varepsilon}{3}$. 

Choose $c_{n + 1} > 0$ big enough such that $$\left\vert \frac{p_n}{q_n} - \llbracket 2c_1, ..., 2c_n, 2c_{n + 1} \rrbracket \right\vert < \frac{1}{2 c_{n + 1} \vert q_n \vert - \vert q_{n - 1} \vert} < \frac{\varepsilon}{3}$$
and let $\frac{p_{n + 1}}{q_{n + 1}} = \llbracket 2c_1, ..., 2c_n, 2c_{n + 1} \rrbracket$.

Set $$\frac{p_{n + 2}}{q_{n + 2}} = \llbracket 2c_1, ..., 2c_n, 2c_{n + 1}, -1 \rrbracket = [-2c_1, 2c_2, ..., -2c_n, 2c_{n + 1}, 1].$$

We then have $\vert x - \frac{p_{n + 2}}{q_{n + 2}} \vert < \varepsilon$, since $$\left\vert \frac{p_{n + 1}}{q_{n + 1}} - \frac{p_{n + 2}}{q_{n + 2}} \right\vert < \frac{1}{\vert q_{n + 1} \vert} < \frac{1}{2 c_{n + 1} \vert q_n \vert - \vert q_{n - 1} \vert} < \frac{\varepsilon}{3}$$
by Lemma \ref{lemrel}.
Also, $$S_4(p_{n + 2}, q_{n + 2}) = m + 1 - 1 + 1 - 1 = m$$
by Theorem \ref{mainthm}.
\end{proof}

\newpage

\section{Numerical Values} \label{numericssec}

\begin{tiny}
We include a numerical table of the values of $S(d, c)$ and $S_4(d, c)$ for all $1 \leq c \leq 10$, so that the reader can check the formulas in Theorem \ref{mainthm} for herself. Since $S(d + 2c, c) = S(d, c)$ and $S(-d, c) = - S(d, c)$, as well as $S_4(d + 2c, c) = S_4(d, c)$ and $S_4(-d, c) = - S_4(d, c)$, it suffices to give the values for $1 \leq d < c$ only. 

If a (finite) $\G_\theta$- or $\G^0(2)$-expansion does not exist, we put the symbol "$\times$". If a number repeats itself several times, we write $\overset{n}{...}$ to indicate that the number before will repeat itself $n$-times where we indicated "$...$". E.g., $[2, 2, 2, 2, 4] = [2, \overset{3}{...}, 4]$.
\end{tiny}

\begin{center}
\begin{tabular}{ |c||c|c||c|c||  }
 \hline
 $(d, c)$ & $\G_\theta$-expansion & $S(d, c)$ & $\G^0(2)$-expansion & $S_4(d, c)$ \\
 \hline
 $(1, 2)$   & $\llbracket -2 \rrbracket$ & 1 & $[2]$ & 1 \\
 $(1, 3)$ & $\times$ & $\times$ & $[3]$ & 2 \\
 $(2, 3)$ & $\llbracket - 2, - 2 \rrbracket$ & 2 & $\times$ & $\times$ \\
 $(1, 4)$ & $\llbracket -4 \rrbracket$ & 1 & $[4]$ & $3$ \\
 $(3, 4)$ & $\llbracket -2, -2, -2 \rrbracket$ & $3$ & $[2, -2, 2]$ & $1$ \\
 $(1, 5)$ & $\times$ & $\times$ & $[5]$ & $4$ \\
 $(2, 5)$ & $\llbracket -2, 2 \rrbracket$ & $0$ & $\times$ & $\times$ \\
 $(3, 5)$ & $\times$ & $\times$ & $[2, -2, -1]$ & $0$ \\
 $(4, 5)$ & $\llbracket -2, -2, -2, -2 \rrbracket$ & $4$ & $\times$ & $\times$ \\
 $(1, 6)$ & $\llbracket -6 \rrbracket$ & $1$ & $[6]$ & $5$ \\
 $(5, 6)$ & $\llbracket -2, -2, -2, -2, -2 \rrbracket$ & $5$ & $[2, -2, 2, -2, 2]$ & $1$ \\
 $(1, 7)$ & $\times$ & $\times$ & $[7]$ & $6$ \\
 $(2, 7)$ & $\llbracket 0, -4, -2 \rrbracket$ & $2$ & $\times$ & $\times$ \\
 $(3, 7)$ & $\times$ & $\times$ & $[2, 2, 1]$ & $2$ \\
 $(4, 7)$ & $\llbracket -2, -4 \rrbracket$ & $2$ & $\times$ & $\times$ \\
 $(5, 7)$ & $\times$ & $\times$ & $[2, -2, 3]$ & $2$ \\
 $(6, 7)$ & $\llbracket -2, \overset{4}{...}, -2 \rrbracket$ & $6$ & $\times$ & $\times$ \\
 $(1, 8)$ & $\llbracket -8 \rrbracket$ & $1$ & $[8]$ & $7$ \\
 $(3, 8)$ & $\llbracket -2, 2, 2 \rrbracket$ & $-1$ & $[3, -2, -1]$ & $1$ \\
 $(5, 8)$ & $\llbracket -2, -2, 2 \rrbracket$ & $1$ & $[2, -2, -2]$ & $-1$ \\
 $(7, 8)$ & $\llbracket -2, \overset{5}{...}, -2 \rrbracket$ & $7$ & $[2, -2, 2, -2, 2, -2, 2]$ & $1$ \\
 $(1, 9)$ & $\times$ & $\times$ & $[9]$ & $8$ \\
 $(2, 9)$ & $\llbracket -4, 2 \rrbracket $ & $0$ & $\times$ & $\times$ \\
 $(4, 9)$ & $\llbracket -2, 4 \rrbracket$ & $0$ & $\times$ & $\times$ \\
 $(5, 9)$ & $\times$ & $\times$ & $[2, -4, -1]$ & $0$ \\
 $(7, 9)$ & $\times$ & $\times$ & $[2, -2, 2, -2, -1]$ & $0$ \\
 $(8, 9)$ & $\llbracket -2, \overset{6}{...}, -2 \rrbracket$ & $8$ & $\times$ & $\times$ \\
 $(1, 10)$ & $\llbracket -10 \rrbracket$ & $1$ & $[10]$ & $9$ \\
 $(3, 10)$ & $\llbracket -4, -2, -2 \rrbracket$ & $3$ & $[3, 2, 1]$ & $3$ \\
 $(7, 10)$ & $\llbracket -2, -2, -4 \rrbracket$ & $3$ & $[2, -2, 4]$ & $3$ \\
 $(9, 10)$ & $\llbracket -2, \overset{7}{...} , -2 \rrbracket$ & $9$ & $[2, -2, 2, -2, 2, -2, 2, -2, 2]$ & $1$\\
 \hline
\end{tabular}    
\end{center}

\newpage

\appendix

\section{Proof of Proposition \ref{eiscontfrac}} \label{appproof}

Let $n \geq 1$ be an odd integer and $a_0, ..., a_n$ be integers. Let $A_k = \SmallMatrix{a_k}{(-1)^{k + 1}}{(-1)^k}{0}$ for $k = 0, ..., n$ and write $\SmallMatrix{*}{*}{c_k}{d_k} = A_0 A_1 \cdots A_k$. Moreover, set $A = \SmallMatrix{a}{b}{c}{d} = A_0 A_1 \cdots A_n$.

We want to prove that \begin{equation*}
    \displaystyle\int_z^{A.z} E_2(w) dw = \frac{6}{\pi i} \log \left( \frac{cz + d}{\sgn(c) i} \right)  + \sum_{k = 0}^n (-1)^k a_k + 3 \sum_{k = 1}^n \sgn(c_k d_k).
\end{equation*}

We begin by rewriting the integral as \begin{equation} \label{eqinproof}
    \int_z^{A.z} E_2(w) dw = \frac{6}{\pi i} \int_i^z \frac{c}{cw + d} dw + \int_i^{A.i} E_2(w) dw
\end{equation}
and note that $A_k.i = \frac{a_k i - (-1)^k }{(-1)^k i} = i + (-1)^k a_k$. In particular, it follows from the Fourier expansion of $E_2$ in (\ref{fouriere2}) that $$\displaystyle\int_i^{A_k.i} E_2(w) \; dw = \displaystyle\int_i^{i + (-1)^k a_k} E_2(w) \; dw = (-1)^k a_k.$$

Now we split the line integral from $z$ to $A.z$ into an integral over $n + 1$ paths (this is possible as $E_2(w)$ is a holomorphic function): \begin{align*}
    \int_i^{A.i} E_2 (w) \; dw &= \sum_{k = 0}^{n} \int_{A_0 \cdots A_{k - 1}.i}^{A_0 \cdots A_k.i} E_2 (w) \; dw \\
    &= \sum_{k = 0}^{n} \int_{i}^{A_{k}.i} E_2(w) \vert_2 (A_0 \cdots A_{k - 1}) \; dw \\
    &= \sum_{k = 0}^{n} \int_{i}^{A_{k}.i} E_2 (w) \; dw  + \frac{6}{\pi i} \sum_{k = 1}^n  \int_{i}^{A_{k}.i} \frac{c_{k - 1}}{c_{k - 1}w + d_{k - 1}} \; dw \\
    &= \sum_{k = 0}^n (-1)^k a_k  + \frac{6}{\pi i} \sum_{k = 1}^n  \int_{i}^{A_{k}.i} \frac{c_{k - 1}}{c_{k - 1}w + d_{k - 1}} \; dw.
\end{align*}

As a primitive of $\frac{c_{k - 1}}{c_{k - 1}w + d_{k - 1}} = \frac{-c_{k - 1}}{- c_{k - 1}w - d_{k - 1}}$, we choose $\log \left( \frac{c_{k - 1} w + d_{k - 1}}{\sgn(c_{k - 1}) i} \right)$. Since $$\SmallMatrix{*}{*}{c_{k - 1}}{d_{k - 1}} A_k = \SmallMatrix{*}{*}{c_{k - 1}a_k + (-1)^k d_{k - 1}}{(-1)^{k + 1} c_{k - 1}} = \SmallMatrix{*}{*}{c_k}{d_k},$$ we have $c_{k - 1} A_k.i + d_{k - 1} = \frac{c_k i + d_k}{(-1)^k i}$. Thus, \begin{equation} \label{logeqinproof}
    \sum_{k = 1}^n \int_{i}^{A_{k}.i} \frac{c_{k - 1}}{c_{k - 1}w + d_{k - 1}} dw = \sum_{k = 1}^n \left( \log \left( \frac{c_{k} i + d_k}{\sgn(d_k)} \right) - \log \left( \frac{c_{k - 1} i + d_{k - 1}}{\sgn(c_{k - 1}) i} \right) \right).
\end{equation}

We may reorder the sum on the right hand side of (\ref{logeqinproof}) to obtain  \begin{align*}
    &\frac{6}{\pi i } \log \left( \frac{c_n i + d_n}{\sgn(d_n)} \right) - \frac{6}{\pi i} \log \left( \frac{i}{i} \right) + \frac{6}{\pi i} \sum_{k = 1}^{n - 1} \left( \log \left( \frac{c_k i + d_k}{\sgn(c_k) i} \right) - \log \left( \frac{c_k i + d_k}{\sgn(d_k)} \right) \right) \\
    &= \frac{6}{\pi i } \log \left( \frac{c i + d}{\sgn(d)} \right) + 3 \sum_{k = 1}^{n - 1} \sgn(c_k d_k).
\end{align*}

Hence, with (\ref{eqinproof}), we arrive at \begin{align*}
    \int_z^{A.z} E_2(w) dw &= \frac{6}{\pi i} \log\left( \frac{cz + d}{\sgn(c) i} \right) - \frac{6}{\pi i} \log\left( \frac{ci + d}{\sgn(c) i} \right) + \sum_{k = 0}^n (-1)^k a_k \\
    &+ \frac{6}{\pi i} \log\left( \frac{ci + d}{\sgn(d)}\right) + 3 \sum_{k = 1}^{n - 1} \sgn(c_k d_k) \\
    &= \frac{6}{\pi i} \log\left( \frac{cz + d}{\sgn(c) i} \right) +  \sum_{k = 0}^n (-1)^k a_k + 3 \sum_{k = 1}^{n} \sgn(c_k d_k).
\end{align*}

\end{document}